\documentclass[12pt]{amsart}
\usepackage{amssymb}
\usepackage{color,url,mathrsfs}

\theoremstyle{plain}
\newtheorem{thm}{Theorem}

\newtheorem{lem}[thm]{Lemma}
\theoremstyle{remark}
\newtheorem{rem}{Remark}
\newtheorem{ex}{Example}

\DeclareMathOperator{\rank}{rank}
\DeclareMathOperator{\vol}{vol}
\DeclareMathOperator{\supp}{supp}

\def\dd{d}
\def\cB{\mathcal{B}}

\def\cF{\mathcal{F}}
\def\sB{\mathscr{B}}
\def\sI{\mathscr{I}}
\def\la{\lambda}
\def\bra{\langle}
\def\ket{\rangle}
\def\eps{\epsilon}

\def\Z{\mathbb{Z}}

\def\R{\mathbb{R}}
\def\C{\mathbb{C}}
\def\SS{\mathbb{S}}
\def\bE{\mathbb{E}}
\def\trivial{{\mathbf 1}}
\def\conf{\mathrm{Conf}(S)}

\def\Prob{\mathbb{P}}
\def\ii{\sqrt{-1}}

\def\xila{\mathcal{X}_{\lambda}}

\def\Xilapeps{\Xi_{\lambda, p, \eps}}
\def\Klapeps{K_{\lambda, p, \eps}}
\def\cX{\mathcal{X}}

\def\balln{B_1^{(m)}}
\def\nn{m}

\makeatletter
\def\filename{\texttt{\jobname.tex}} 
\makeatother

\numberwithin{equation}{section}

\title[Local universality of determinantal point processes
on Riemannian manifolds]{Local universality of determinantal point processes on
Riemannian manifolds}
\author{Makoto \textsc{Katori}}
\address{Department of Physics, Faculty of Science and Engineering, Chuo University, Kasuga, Bunkyo-ku, Tokyo 112-8551,
Japan} 
\email{katori@phys.chuo-u.ac.jp}
\author{Tomoyuki \textsc{Shirai}}
\address{Institute of Mathematics for Industry, 
Kyushu University, 744 Motooka, Nishi-ku, Fukuoka 819-0395, Japan}  
\email{shirai@imi.kyushu-u.ac.jp}

\begin{document}
\begin{abstract} 
We consider the Laplace-Beltrami operator $\Delta_g$ on 
a smooth, compact Riemannian manifold $(M,g)$ 
and the determinantal point process $\xila$ on $M$ associated with the spectral
 projection of $-\Delta_g$ onto the subspace corresponding to the eigenvalues up to
$\la^2$. 
We show that the pull-back of $\xila$ by the exponential map
 $\exp_p : T_p^*M \to M$ under a suitable scaling converges
 weakly to the universal determinantal point process on
 $T_p^* M$ as $\la \to \infty$. 
\end{abstract}
\maketitle

\section{Introduction}
Let $(M,g)$ be a smooth, compact, Riemannian manifold
of dimension $\nn$ with no boundary.  
We fix an orthonormal basis $\{\phi_i\}_{i \ge 0}$
of eigenfunctions of 
the Laplace-Beltrami operator $\Delta_g$ acting on $L^2(M)
:= L^2(M, \vol_g)$:  
\[
 -\Delta_g \phi_i = \la_i^2 \phi_i, \quad
  \bra \phi_i, \phi_j \ket_{L^2(M)} =
  \delta_{ij},  
\]
with $0 = \la_0^2 \le \la_1^2 \le \la_2^2 \le \cdots
\nearrow + \infty$. Here $\{\la_i\}_{i=0}^{\infty}$
are the eigenvalues of $\sqrt{-\Delta_g}$. 
We denote the eigenspace corresponding to an eigenvalue $\la_i$ by
$W_{\la_i}$.  
The projection operator $E_{\la}$ on $L^2(M, \vol_g)$ onto
the closed subspace $W_{\le \la } := \bigoplus_{\la_i
\le \la} W_{\la_i}$ admits the following integral kernel 
\begin{equation}
 E_{\la}(x,y) = \sum_{\la_i \le \la} \phi_i(x)
 \overline{\phi_i(y)} \quad (x, y \in M). 
\label{eq:Elambda}
\end{equation}
The projection kernel $E_{\la}(x,y)$ is the reproducing kernel
of $W_{\le \la}$ and thus defines a determinantal
point process (DPP) $\xila$ on $M$, which
is a random simple point configuration on $M$ whose 
$n$-point correlation function with respect to $\vol_g$ is given by 
\[
 \rho_n(x_1,x_2,\dots, x_n) 
= \det( E_{\la}(x_i,x_j) )_{i,j=1}^n. 
\]
In particular, the $1$-point correlation function, the
density of points, is 
\[
 \rho_1(x) = E_{\la}(x,x). 
\]
See Section~\ref{sec:DPP} for the definition of DPP. 

The number of points in $\xila$ on $M$ is
equal to the eigenvalue counting function given by 
\[
N(\la) = \sum_{\la_i \le \la} 1 = \rank E_{\la} 
= \int_M E_{\la}(x,x) \vol_g(\dd x). 
\]
Since $\{\la_i\}_{i=0}^{\infty}$ are the eigenvalues of
$\sqrt{-\Delta_g}$, it is known as the classical Weyl law
(cf. \cite{H68}) that 
\begin{equation}
 N(\la) \sim \frac{\la^{\nn}}{(2\pi)^{\nn}} |\balln| \vol_g(M) \quad (\la
 \to \infty),  
\label{eq:classicalWeyl}
\end{equation}
where $|\balln|$ is the volume of a unit ball in $\R^{\nn}$, i.e., 
$|\balln| = \pi^{\nn/2}/\Gamma(\nn/2+1)$. 
This means that the points in $\cX_{\la}$ on $M$ become dense as $\la \to \infty$. 

\begin{ex}
When $M=\SS^1$, for every $\la>0$, the DPP associated with
$E_{\la}$ is the random eigenvalues of Circular Unitary
Ensemble (CUE) of size $N(\la)$ (cf. \cite{F10}). More generally, when $M = \SS^{\nn}$,
the corresponding DPPs are called \textit{harmonic ensembles} on $\SS^{\nn}$ 
 (cf. \cite{KS19A}). These point processes are homogeneous
 in the sense that they are invariant
 under the $\mathrm{O}(\nn)$-action. 
\end{ex}

The quantum ergodicity theorem originated by Shnirel'man
\cite{Sh74, Sh93} and also studied in \cite{CdV85, Zel87} 
states that if the geodesic flow on $M$ is ergodic then 
$N(\la)^{-1}E_{\la}(x,x) \vol_g(\dd x)$ converges weakly to
$\vol_g(\dd x)$ as $\la \to \infty$, 
in other words, so does the normalized first correlation
measure of the DPP $\xila$. 
This theorem describes the global behavior of random points of the DPP on $M$. 

In this paper, we focus on the local statistics of points in
the DPP by taking a scaling as in \eqref{eq:Xilambda} below
so that we define a DPP $\Xi_{\la,p}$ on the cotangent space 
$T_p^*M$ by taking the pull-back of
the DPP $\xila$ on $M$ by the exponential map. 

We denote the Riemannian metric on $T_p^* M$ 
by $\bra \cdot, \cdot \ket_{g_p} : T_p^* M \times T_p^*
M \to \R$ and the corresponding norm by $|\cdot|_{g_p}$.  
Here $|\xi|_{g_p}$ is the same as the principal symbol of
$\sqrt{-\Delta_g}$ locally given by 
\[
|\xi|_{g_p} = \Big(\sum_{i,j=1}^{\nn} g^{ij}(p) \xi_i \xi_j
 \Big)^{1/2}, 
\]
and $(g^{ij}(p))_{i,j=1}^{\nn}$ is the inverse matrix $g_p^{-1}$
of $g_p = (g_{ij}(p))_{i,j=1}^{\nn}$. 
The so-called pointwise Weyl law can be expressed as follows: as $\la
\to \infty$, 
\begin{align}
E_{\la}(x,x)
&= \frac{1}{(2\pi)^{\nn}} \int_{|\xi|_{g_x} < \la}
 \frac{\dd\xi}{\sqrt{\det g_x}} +
R_{\la}(x) \label{eq:weyllaw}\\
&= \frac{|\balln|}{(2\pi)^{\nn}} \la^{\nn} + R_{\la}(x) 
\nonumber
\end{align}
with uniform bound $\sup_{x \in M} |R_{\la}(x)| \le C
\la^{\nn-1}$ \cite{H68}, which leads to the classical Weyl law 
\eqref{eq:classicalWeyl}. 

Since $M$ is compact, the injectivity radius $\mathrm{inj}^*(M)$ is
positive, i.e., 
the exponential map $\exp_p : T_p^* M \to M$ is injective on
the subset $\{\xi \in T_p^*M
: |\xi|_{g_p} < \mathrm{inj}^*(M)\}$ for any $p \in M$. 
We fix a point $p \in M$ and positive $\eps < \mathrm{inj}^*(M)$. 
Let $B_{\eps}$ be the open ball of radius $\eps$ in $T_p^*
M$ centered at the origin and 
denote the image $\exp_p (B_{\eps})$ by $\cB_{p,\eps}$. 
For  $\la >0$, we define a point process $\Xilapeps$ on the cotangent space
$T_p^* M$ by 
\begin{equation}
\Xilapeps := \sum_{x \in \xila \cap \cB_{p,\eps}}
\delta_{\la \exp_p^{-1}(x)}, 
\label{eq:Xilambda} 
\end{equation}
which defines the pull-back of $\xila$ restricted on $\cB_{p,\eps}$ by the exponential
map and is scaled by $\la$. 
Here, we identified $\xila$ with a subset in $M$ (see
Section~\ref{sec:defDPP}). 
It turns out again to be a DPP on $T_p^* M$ (see Lemma~\ref{lem:DPP}).  

Our main assertion in this paper is the following. 

 \begin{thm}\label{thm:main}
 As $\la \to \infty$, the point process $\Xilapeps$
converges weakly to the DPP $\Xi_p$ on $T_p^* M$ associated 
 with the kernel 
 \begin{equation}
 K_{g_p}^{(\nn)}(u,v) 
= \frac{1}{(2\pi |u-v|_{g_p})^{\nn/2}}
 J_{\nn/2}(|u-v|_{g_p})
\label{eq:kernelKgp}  
 \end{equation}
and the reference measure $\vol_{G^{(p)}}$, where $J_{\alpha}(x)$ is the Bessel function of the first
kind defined by 
\[
 J_{\alpha}(x) = \sum_{k=0}^{\infty} \frac{(-1)^k}{k!
 \Gamma(k+\alpha+1)} \Big(\frac{x}{2}\Big)^{2k+\alpha}
\]
and $\vol_{G^{(p)}}$ is the Riemannian measure on $T_p^*M$ with
  respect to the constant Riemannian metric $G^{(p)}=(G^{(p)}_u)_{u \in T_p^*M}$
  such that $G^{(p)}_u = (\exp_p^* g)_0$ for every $u \in T_p^*M$. 
\end{thm}
We remark that the limiting DPP $\Xi_p$ does not
depend on $\eps>0$.

We consider the following correlation kernel on $\R^m$, 
\begin{align*}
K^{(\nn)}(u,v)
&:= \frac{1}{(2\pi |u-v|)^{\nn/2}}
 J_{\nn/2}(|u-v|) \\
&= \frac{1}{(2\pi)^{\nn}} 
\int_{|\xi|<1}  e^{\ii (u-v, \xi)} \dd\xi, 
\end{align*}
where $(\cdot, \cdot)$ (resp. $|\cdot |$) is the standard
inner product (resp. norm) on $\R^{\nn}$. The DPP on $\R^m$ associated with $K^{(\nn)}(u,v)$ 
is invariant under the action of the Euclidean motion
group. When $n=1$, $K^{(1)}(u,v)$ coincides with the sinc kernel 
\[
 K^{(1)}(u,v) = \frac{\sin (u-v)}{\pi (u-v)},  
\]
which is the reproducing kernel of the classical
Paley-Wiener space (see also Example~\ref{ex:RKHS} for $K^{(\nn)}$ given in Section~\ref{sec:DPPandrkhs}). 
It is well known that 
the point process of eigenvalues of CUE (also GUE) under
suitable scaling converges to the  
DPP associated with the sinc kernel. 
This DPP is also one of the most important examples of the
class of DPPs associated with de Branges spaces discussed in \cite{BS17}. 
In \cite{KS19A}, we proved a special case of Theorem~\ref{thm:main} when $M = \SS^{\nn}$ by using spherical harmonics. 
Theorem~\ref{thm:main} can be regarded as a generalization of these results to compact Riemannian manifolds. 
For the proof of Theorem~\ref{thm:main}, the pointwise Weyl law \eqref{eq:weyllaw} plays a central role.  

Theorem~\ref{thm:main} shows the local universality of DPPs on Riemannian
manifolds. This type of universality has been discussed 
as the asymptotic local structure of Szeg\H{o} kernels, which is
used to analyze random spherical harmonics and random section
of holomorphic line bundles over a compact K\"{a}hler
manifold. The former corresponds to the Euclidean class
(real case) while the latter does the Heisenberg class (complex case)
(cf. \cite{BSZ00, Zel00, Zel09}). The terms ``Euclidean'' and ``Heisenberg''
are related to representations of the Euclidean and Heisenberg motion groups. 
The result in this paper falls in the
Euclidean class in this terminology. 

Theorem~\ref{thm:main} can also be generalized to the case where the
spectral projections of Laplace-Beltrami operators are replaced by those of 
general elliptic operators. 

\section{Determinantal point processes}\label{sec:DPP}

For the necessary background for determinantal point
   processes, see e.g. \cite{Macchi, ShirTaka0,
   ShirTaka1, ShirTaka2, Soshnikov, HoughEtAl,KS19B}.

\subsection{Definition}\label{sec:defDPP}

Let $S$ be a locally compact Hausdorff space with countable
base. A configuration $\Xi$ on $S$ is a non-negative
integer-valued Radon measure and it can be expressed as $\Xi
= \sum_{i} \delta_{x_i} \ (x_i \in S)$. 
We denote by $\conf$ the totality of configurations on 
$S$, which we call a configuration space over $S$. 
An element $\Xi$ of $\conf$ is sometimes 
regarded as an at most countable subset in $S$ without
accumulation, possibly with multiple points. Thus, $\Xi(A)$
is equal to
the number of points in $A \in \sB(S)$ with counted
multiplicity, 
where $\sB(S)$ is the totality of all bounded (i.e.,
relatively compact) sets in $S$. 
The configuration space $\conf$ equipped with vague topology
turns out to be a Polish space, i.e, a complete, separable
metrizable space. 
We equip the configuration space $\conf$ with the Borel
structure with respect to this topology, which coincides with the
Borel structure generated by the mapping $\conf \ni \Xi \mapsto \Xi(A) \in
S$ for all bounded $A \in \sB(S)$. 
A point process on $S$ is a $\conf$-valued random variable 
$\Xi = \Xi_{\omega}$ 
defined on a probability space $(\Omega, \cF, \Prob)$. 
If $\Xi(\{x\}) \le 1$ for every $x \in S$ a.s., then $\Xi$ is called a
simple point process. In this case, 
by identifying $\Xi$ with its support, 
we use the notation $x \in \Xi$ meaning that $\Xi(\{x\})=1$. 

We fix a Radon measure $\nu$ on $S$ as a reference measure. 
A symmetric measure $\nu_n$ on $S^n$
is called the $n$-th {\it correlation measure} if it satisfies
\[
\bE 
\left[
\prod_{i=1}^p \frac{\Xi(A_i)!}
{(\Xi(A_i)-k_i)!} \right]
=\nu_n(A_1^{k_1} \times \cdots
\times A_p^{k_p})
\]
for any disjoint bounded sets $A_1, \dots, A_p \in \sB(S)$ and
any $k_1, \dots, k_p \in \Z_{\ge 0} := \{0,1,2,\dots\}$  
with $\sum_{i=1}^p k_i = n$. 
If $\nu_n$ is absolutely continuous 
with respect to the product measure $\nu^{\otimes n}$,
the Radon-Nikodym derivative
$\rho_n := d\nu_n/ d\nu^{\otimes n}$ is called 
the {\it $n$-point correlation function}
with respect to the reference measure $\nu$;
\[
\nu_n(dx_1 \dots dx_n)
=\rho^n(x_1, \dots, x_n) 
\nu^{\otimes n}(dx_1 \dots dx_n).
\]

Let $\sI_{1}(S,\nu)$ be the ideal of trace class operators
$K \colon L_2(S,\nu)\to L_2(S,\nu)$; we denote the 
$\sI_{1}$-norm of the operator $K$ by $||K||_{{\mathscr
I}_{1}}$.
Let  $\sI_{1,  \mathrm{loc}}(S,\nu)$ be the space of operators
$K\colon L_2(S,\nu)\to L_2(S,\nu)$ 
such that $\trivial_A K \trivial_A\in{\mathscr I}_1(S,\nu)$
for any bounded Borel subset $A \subset S$, 
where $\trivial_A$ is the indicator function of a set $A$. 
Such an operator $K$ is called a locally trace class
operator. 
We endow the space ${\mathscr I}_{1, \mathrm{loc}}(S,\nu)$
with a countable family of semi-norms
$\|\trivial_A K \trivial_A \|_{{\mathscr I}_1}$
where $A$ runs through an exhausting family $A_n$ of bounded
sets, i.e., $A_n$ is increasing and $\bigcup_{n=1}^{\infty} A_n = S$. 
A locally trace class operator $K$ admits a kernel
(cf. \cite{GY05, KS19B}), 
for which, slightly abusing notation, we use the same symbol $K$.

A point process is called a \textit{determinantal point
process associated with $K$ and $\nu$}  
if there exists an operator
$K\in{\mathscr I}_{1,  \mathrm{loc}}(S,\nu)$ such that for
any bounded measurable function $h$, for which $h-1$ is
supported in a bounded set $A$, 
we have
\begin{equation}
\label{eq1}
\bE \Psi_h 
=\det\biggl(1+(h-1)K \trivial_A \biggr),
\end{equation}
where $\Psi_h(\Xi) = \prod\limits_{x \in \Xi} h(x)$ 
for $\Xi \in \conf$.
The Fredholm determinant in~\eqref{eq1} is well-defined since
$K\in \sI_{1, \mathrm{loc}}(S,\nu)$. For example, if $K$ is a
positive contraction operator $K \in \sI_{1,
\mathrm{loc}}(S,\nu)$, then there exists a DPP associated
with $K$ and $\nu$. 
The equation (\ref{eq1}) determines the law of the DPP uniquely (\cite{ShirTaka0, ShirTaka1, Soshnikov}). 
For the DPP associated with $K$, the $n$-th correlation
function with respect to $\nu$ is given by 
\[
 \rho_n(x_1,\dots, x_n) = \det(K(x_i,x_j))_{i,j=1}^n. 
\]
$K(x,y)$ is often called the correlation kernel and $\nu$ the reference measure. 
When $S=\R^{\nn}$, if $\nu$ is the Lebesgue measure and $K(x,y) = k(x-y)$ for
some $k$, then the law of the DPP associated with $K$ and
$\nu$ is invariant under the action of the Euclidean motion group. 

Weak convergence for DPPs is characterized by the
convergence of operators (cf. Proposition 3.10 in
\cite{ShirTaka1}) as follows.     
\begin{lem}\label{lem:weakconv}
Let $\Xi_n$ (resp. $\Xi$) be a DPP on $S$ associated with $K_n$
 (resp. $K$) and $\nu$. 
Suppose $K_n$ converges $K$ in $\sI_{1, \mathrm{loc}}(S,\nu)$ as
 $n \to \infty$. 
Then $\Xi_n$ converges weakly to $\Xi$ as $n \to
 \infty$. In particular, if the kernel $K_n(x,y)$ converges to $K(x,y)$
 uniformly on any compact set in $S \times S$, then the convergence
 above takes place. 
\end{lem}

\subsection{DPPs associated with reproducing kernel Hilbert spaces}\label{sec:DPPandrkhs}

Let $S$ be a non-empty subset and $\cF(S)$ be a linear space
of functions on $S$, i.e., $\cF(S) := \{f : S \to \C\}$. 
A subspace $H$ of $\cF(S)$ is called a reproducing
kernel Hilbert space (RKHS) if $H$ is endowed with an inner product
$\bra \cdot, \cdot \ket_H$ which makes $H$ a Hilbert space 
and the evaluation functional $E_s : H \to \C$ defined by 
$E_s(f) := f(s)$ is bounded for every $s \in S$. 
By the Riesz representation theorem, for each $s \in S$, there
exists a unique element $k_s \in H$ 
such that $E_s(f) = \bra f, k_s \ket_H = f(s)$. 
We define a kernel $K : S \times S \to \C$ by
\[
 K(s,t) := k_t(s) = \bra k_t, k_s \ket_H, 
\]
which is called the reproducing kernel for $H$ 
(see \cite{Aro50} for more details about RKHS).
The integral operator $K$ with kernel $K(s,t)$ 
defines an orthogonal projection onto $H$. 
Therefore, the DPP is associated with reproducing kernel
$K(s,t)$, or equivalently, RKHS $H$.  
\begin{ex}\label{ex:RKHS}
(1) For a given $a>0$, 
\[
 \mathrm{PW}_a := \{f \in C(\R) : \supp \widehat{f} \subset [-a,a]\}
\]
is called a Paley-Wiener space or the space of band-limited
 functions. Here $\widehat{f}$ is the Fourier transform of
 $f$ defined as 
\[
\widehat{f}(\xi) := \int_{\R^{\nn}} f(x)e^{-\sqrt{-1}(x, \xi)} dx.  
\]
The corresponding reproducing kernel $K_a$ is given by 
\[
 K_a(x,y)= \frac{\sin a(x-y)}{\pi(x-y)}
\]
and the corresponding DPP is the limiting DPP obtained from
 CUE (also GUE) eigenvalues. \\ 
(2) A generalized Paley-Wiener space is similarly defined as
 follows: for a bounded Borel set $\Omega \subset \R^{\nn}$, 
\[
 \mathrm{PW}_{\Omega} := \{f \in C(\R^{\nn}) : \supp \widehat{f}
 \subset \overline{\Omega}\}. 
\]
When $\Omega = \balln \subset \R^{\nn}$, the corresponding
 reproducing kernel is $K^{(\nn)}(x,y)$ which appeared in 
Theorem~\ref{thm:main}. \\
(3) Let $(M,g)$ be a compact, smooth, Riemannian manifold
 and $\Delta_g$ be the Laplace-Beltrami operator on $L^2(M, \vol_g)$. 
We denote the resolution of the identity for $\Delta_g$ 
by $\{E(A) : A \in \sB(\R)\}$. 
Then the integral operator $E_{\la}$ with kernel $E_{\la}(x,y)$ 
given in \eqref{eq:Elambda} 
coincides with the projection $E([0, \la^2])$ and $W_{\le \la}$ 
turns out to be a RKHS admitting the reproducing kernel
 $E_{\la}(x,y)$.  
\end{ex}

\section{Proof of the main theorem} \label{sec:proof}
We define $\phi_{\la} : T_p^* M \to M$ by $\phi_{\la}(u) =
 \exp_p(u/\la)$ for $u \in T_p^* M$. 
For $u, v \in T_p^* M$, 
we write $U_{\la} = \phi_{\la}(u)$ and $V_{\la} = \phi_{\la}(v)$. 
We consider the kernel 
\begin{equation}
\Klapeps(u,v) = \frac{1}{\la^{\nn}} E_{\la}(U_{\la},
 V_{\la})
\trivial_{\cB_{p,\eps}}(U_{\la})
\trivial_{\cB_{p,\eps}}(V_{\la}). 
\label{eq:Klambda} 
\end{equation}
We have the following. 
\begin{lem}\label{lem:DPP}
The scaled point process $\Xilapeps$ defined by \eqref{eq:Xilambda} is the DPP on $T_p^* M$ associated with the kernel
$\Klapeps(u,v)$ of \eqref{eq:Klambda} and $\la^m \phi_{\la}^* \vol_g$. 
\end{lem}
\begin{proof}
We note that $\cX_{\la}|_{\cB_{p,\eps}}$ is the DPP
 associated with
the kernel $E_{\la}(x,y)
 \trivial_{\cB_{p,\eps}}(x)\trivial_{\cB_{p,\eps}}(y)$ and
 the reference measure $\vol_g$. Then the pull-back $\phi_{\la}^* \cX_{\la}|_{{\cB_{p,\eps}}}$ is the DPP
 associated with the kernel $E_{\la}(\phi_{\la}(u),\phi_{\la}(v)) 
\trivial_{\cB_{p,\eps}}(\phi_{\la}(u))\trivial_{\cB_{p,\eps}}(\phi_{\la}(v))$
 and $\phi_{\la}^* \vol_g$ since $\phi_{\la}|_{B_{\eps}} : B_{\eps} \to
 \cB_{p,\eps}$ is a diffeomorphism. 
The law of this DPP is the same as that of the DPP
 associated with the kernel \eqref{eq:Klambda} 
and $\la^m \phi_{\la}^* \vol_g$ through the 
measure change by the factor $\la^m$ (cf. \cite[Section 2.3]{KS19B}). 
\end{proof}
We remark that since $(d\phi_{\la})_u = \la^{-1} (d\phi_1)_{u/\la}$, 
the pull-back of the Riemannian metric $g$ on $M$ is expressed as 
\begin{align*}
\la^2 (\phi_{\la}^*g)_u 
= (\phi_1^* g)_{u/\la}
\end{align*}
for $u \in T_p^*M$. 
Therefore, $\la^m \phi_{\la}^* \vol_g$ is equal to 
the Riemannian measure with respect to $(\phi_1^*g)_{\cdot/\la}$. 
For the proof of Theorem~\ref{thm:main}, we appeal to the pointwise Weyl
law \eqref{eq:weyllaw}, which gives an off-diagonal asymptotics for the
projection kernel $E_{\la}(x,y)$ as $\la \to \infty$ as
follows: 
if $x$ is close enough to $y$, i.e., 
$x \in \cB_{y,\eps}$ with $\eps < \mathrm{inj}^*(M)$, 
then 
\begin{align}
E_{\la}(x,y)
&= \frac{\la^{\nn}}{(2\pi)^{\nn}} \int_{|\xi|_{g_y} < 1} 
e^{\ii \la \psi(x,y,\xi)} 
\frac{\dd\xi}{\sqrt{\det g_y}} 
+ R_{\la}(x,y), 
\label{eq:integral} 
\end{align}
where $\psi(x,y,\xi)$ is a phase function which is adapted, 
in H\"{o}rmander's terminology \cite{H68}, 
to the principal symbol $|\xi|_{g_y}$ of $\sqrt{-\Delta_g}$,  
vanishing on the diagonal $x=y$. 
This type of asymptotics for the spectral function was
initiated by H\"ormander \cite{H68} as an application of the
theory of pseudo-differential operators and recovers the
classical Weyl law \eqref{eq:classicalWeyl}. 
The choice of a phase function is not unique, and one can
take 
\begin{equation}
 \psi(x,y,\xi) = \bra \exp_y^{-1}(x), \xi \ket_{g_y} 
\label{eq:phase-function}
\end{equation}
in a coordinate-independent way \cite{Zel09b, CH15}. 
Indeed, the integral on the right-hand side of
\eqref{eq:integral} 
with \eqref{eq:phase-function} is taken over the cotangent fiber $T_y^*M$ and 
it is coordinate-independent since the measure $\dd\xi /\sqrt{\det g_y}$ is the quotient of the canonical symplectic form
$\dd\xi \wedge dy$ on $T^*M$ by the Riemannian volume form $\sqrt{\det g_y} dy$ on $M$. 
There are many papers estimating the remainder term
$R_{\la}(x,y)$. 
From \cite[Theorem 2]{CH15}, the remainder term is uniformly estimated as
follows. 
\begin{thm}[\cite{H68, CH15, CH18}]\label{thm:estimate} 
We assume \eqref{eq:phase-function}. Then, 
for any fixed $r > 0$, as $\la \to \infty$, 
\[
\sup_{d_g(x,y) < r/ \la} 
| R_{\la}(x,y) | = O(\la^{\nn-1}), 
\]
where $d_g(x,y)$ is the Riemannian distance. 
\end{thm}

Before giving a proof of the main theorem, 
we see a generalization of the following formula (cf. \cite{KS19B})
\begin{equation}
\frac{1}{(2\pi)^{\nn/2}} 
\int_{|\omega| < 1} e^{\ii (\eta,
 \omega)} d \omega = F_{\nn/2}(|\eta|),   
\label{eq:fourier-ball} 
\end{equation}
where $F_{\alpha}(t) = J_{\alpha}(t)/ t^{\alpha}$ for $\alpha >0$.
\begin{lem}\label{lem:fourier-sphere}
 Let $\nn = \dim M$. For $\eta \in T_p^*M$, 
\[
\frac{1}{(2\pi)^{\nn/2}} 
\int_{|\xi|_{g_p} < 1} e^{\ii \bra \eta,
 \xi\ket_{g_p}} \frac{d \xi}{\sqrt{\det g_p}}
= F_{\nn/2}(|\eta|_{g_p}).    
\]
\end{lem}
\begin{proof}
We note that 
\[
 \bra \eta, \xi \ket_{g_p} = (g_p^{-1/2} \eta, 
g_p^{-1/2} \xi),  
\]
where $g_p^{-1/2}$ is the positive definite square root of the inverse matrix $g_p^{-1}$. In 
particular, $|\eta|_{g_p} = |g_p^{-1/2} \eta|$.  
From \eqref{eq:fourier-ball}, by change of variables $\omega = g_p^{-1/2} \xi$, 
we have 
\begin{align*}
F_{\nn/2}(|\eta|_{g_p}) 
&= \frac{1}{(2\pi)^{\nn/2}}
\int_{|\omega| < 1} e^{\ii (g_p^{-1/2}\eta, \omega)}
 d\omega \\
&= \frac{1}{(2\pi)^{\nn/2}}
\int_{|\xi|_{g_p} < 1} e^{\ii \bra \eta,
 \xi \ket_{g_p}} \frac{d \xi}{\sqrt{\det g_p}}. 
\end{align*}
We obtain the assertion. 
\end{proof}

\begin{rem} We have a similar formula 
\[
\frac{1}{(2\pi)^{\nn/2}} 
\int_{|\xi|_{g_p} = 1} e^{\ii \bra \eta,
 \xi\ket_{g_p}} \frac{d \xi}{\sqrt{\det g_p}}
= F_{(\nn-2)/2}(|\eta|_{g_p}). 
\]
\end{rem}
We need one more fact for the local behavior of the Riemannian
distance function. 

\begin{lem}\label{lem:limit} 
For $u, v \in T_p^*M$, let $c_1$ and $c_2$ be $C^1$ curves in $M$
 such that $c_1(0)=c_2(0)=p$, $c_1'(0)=u$ and
 $c_2'(0)=v$. Then,  
\[
\lim_{t \to 0+} \frac{d_g(c_1(t), c_2(t))}{t} =
 |u-v|_{g_p}. 
\] 
\end{lem}
\begin{proof}
See Corollary 3.1 in \cite{CFL} for instance. 
\end{proof}

Now we are in a position to give a proof of the main theorem. 
\begin{proof}[Proof of Theorem~\ref{thm:main}]
It suffices to show that the DPP associate with 
$\Klapeps(u,v)$ and $\la^m \phi_{\la}^* \vol_g$ converges as
 $\la \to \infty$. 
Suppose $d_g(x,y)$ is small enough. 
First we note that there exists $\zeta \in T_y^*M$ such that 
$|\zeta|_{g_y}=1$ and $\exp_y^{-1}(x) = d_g(x,y) \zeta$. 
By using Lemma~\ref{lem:fourier-sphere}, 
we see that 
\begin{align*}
& \frac{1}{(2\pi)^{\nn/2}} \int_{|\xi|_{g_y} < 1} 
e^{\ii \bra\exp_y^{-1}(x), \xi\ket_{g_y}}
\frac{\dd\xi}{\sqrt{\det g_y}}
=F_{\nn/2}(d_g(x,y)). 
\end{align*}
From Lemma~\ref{lem:DPP} with \eqref{eq:Klambda}, 
\eqref{eq:integral} with \eqref{eq:phase-function}, 
Theorem~\ref{thm:estimate} and
 Lemma~\ref{lem:fourier-sphere}, 
as $\la \to \infty$, we have 
\begin{align*}
\Klapeps(u,v) 
&= \frac{1}{(2\pi)^{\nn}} \int_{|\xi|_{g_{V_{\la}}} < 1} e^{\ii \la \bra
 \exp_{V_{\la}}^{-1}(U_{\la}),  
\xi \ket_{g_{V_{\la}}}} \frac{\dd\xi}{\sqrt{\det g_{V_\la}}} 
\trivial_{\cB_{p,\eps}}(U_{\la}) \trivial_{\cB_{p,\eps}}(V_{\la}) + O(\la^{-1})\\
&= \frac{1}{(2\pi)^{\nn/2}} 
F_{\nn/2}(\la d_g(U_{\la}, V_{\la})) 
\trivial_{\cB_{p,\eps}}(U_{\la})
 \trivial_{\cB_{p,\eps}}(V_{\la}) + O(\la^{-1}). 
\end{align*}
We note that $\lim_{t \to 0} F_{\alpha}(t) =
 2^{-\alpha} \Gamma(\alpha+1)^{-1}$ and so $F_{\alpha}(t)$ is a bounded
 continuous function on $\R$.  
Since  $\la d_g(U_{\la}, V_{\la}) \to |u-v|_{g_p}$ by 
Lemma~\ref{lem:limit} 
and $\trivial_{\cB_{p,\eps}}(U_{\la}) \trivial_{\cB_{p,\eps}}(V_{\la})$ is
 equal to $1$ for any sufficiently large $\la$, we have 
\begin{align*}
\Klapeps(u,v)
&\to \frac{1}{(2\pi)^{\nn/2}}
 F_{\nn/2}(|u-v|_{g_p})
\end{align*}
uniformly on compacts in $T_p^* M$. 
From the remark after Lemma~\ref{lem:DPP}, 
the reference measure is the Riemannian measure with
 respect to $(\phi_1^* g)_{\cdot /\la}$ and the
 Radon-Nikodym derivative relative to 
the Riemannian measure with respect to 
$(\phi_1^* g)_0$ is uniformly close to $1$ on any compact set as $\la \to \infty$. 
Therefore, it follows from Lemma~\ref{lem:weakconv} 
that the scaled point process 
$\Xilapeps$ converges weakly to the DPP associated with the kernel 
$K_{g_p}^{(\nn)}(u,v)$ given by \eqref{eq:kernelKgp} and the
 reference measure given by the Riemannian measure with respect to
the constant metric $(\phi_1^* g)_0$. The proof is completed. 
\end{proof}

\section{Concluding remarks}
We have seen the local universality of DPPs on Riemannian
manifolds. From this discussion, we came to several other
questions. 

\begin{enumerate}
\item What is the universality when we consider the
      Heisenberg case in Zelditch's terminology? 
      We only discussed the Euclidean case in this article. 
      One can expect that the Bergman kernel is
      involved as in \cite{BSZ00, Zel00, Zel09, KS19A}.  

\item We dealt with Laplace-Beltrami operators 
      corresponding to the principal symbol $|\xi|_{g_y}$. 
      What is the local universality result when we consider more general
      DPPs associated with the spectral  projections of
      elliptic differential operators possibly with potentials?  

\item In this paper we have considered a point process 
$\Xi_{\lambda, p, \epsilon}$ on
$T^*_p M$ at each `point' $p \in M$.
We expect that the collection 
$\Xi_{\lambda, \epsilon}= \{ \Xi_{\lambda, p, \epsilon} \}_{p \in M}$
will be regarded as a `random field' 
on the cotangent bundle $T^* M=\{T_p^* M\}_{p \in M}$. 
Theorem 1 determines the limit
$\Xi_{\lambda, p, \epsilon} \to \Xi_p$ in $\lambda \to \infty$.
How can we describe the limiting random field $\Xi_{\lambda, \epsilon} \to \Xi$ in
$\lambda \to \infty$?

\end{enumerate}

\vskip 1cm
{\bf Acknowledgments.}
 This research was supported by JSPS KAKENHI
Grant Number 18H01124 and 19K03674, and also partially supported 
by JSPS KAKENHI Grant Number 16H06338, 20H00119, 20K20884,
and 21H04432.  


\end{document}